\documentclass[10pt,a4paper,twoside,reqno]{amsart}

\usepackage[a4paper,textwidth=16cm,textheight=21cm,centering]{geometry}
\usepackage{amssymb}
\usepackage{amsmath}
\usepackage{mathrsfs}
\usepackage{amsthm}
\usepackage{amsfonts}
\usepackage{mathtools}
\IfFileExists{newtxtext.sty}{%
  \IfFileExists{newtxmath.sty}{%
    \usepackage{newtxtext}\usepackage{newtxmath}%
  }{\usepackage{txfonts}}%
}{\usepackage{txfonts}}
\usepackage{enumerate}
\usepackage{enumitem}
\usepackage{latexsym}
\usepackage{color}
\usepackage{microtype}
\usepackage{indentfirst}

\usepackage[
colorlinks=true,
citecolor=red,
linkcolor=blue,
urlcolor=blue,
hypertexnames=false
]{hyperref}

\makeatletter
\def\subsection{\@startsection{subsection}{2}%
	\z@{.7\linespacing\@plus.2\linespacing}%
	{.5\linespacing}%
	{\normalfont\bfseries}}
\makeatother
\allowdisplaybreaks
\newtheoremstyle{boldplain}%
{3pt}{3pt}{\itshape}{}%
{\bfseries\upshape}{.}{0.5em}%
{\thmname{\bfseries #1}\thmnumber{\bfseries\ #2}%
 \thmnote{\bfseries\ (#3)}}

\newtheoremstyle{bolddefinition}%
{3pt}{3pt}{\normalfont}{}%
{\bfseries\upshape}{.}{0.5em}%
{\thmname{\bfseries #1}\thmnumber{\bfseries\ #2}%
 \thmnote{\bfseries\ (#3)}}

\theoremstyle{boldplain}
\newtheorem{theorem}{Theorem}[section]
\newtheorem{proposition}[theorem]{Proposition}
\newtheorem{lemma}[theorem]{Lemma}

\theoremstyle{bolddefinition}
\newtheorem{definition}[theorem]{Definition}
\newtheorem{remark}[theorem]{Remark}

\numberwithin{equation}{section}
\newtheorem{question}{Question}[section]
\newcommand{\R}{\mathbb R}
\newcommand{\Sph}{\mathbb S^{n-1}}
\newcommand{\sig}{\sigma}
\newcommand{\1}{\mathbf 1}
\newcommand{\dd}{\,d}
\newcommand{\abs}[1]{\left|#1\right|}
\newcommand{\norm}[1]{\left\lVert#1\right\rVert}
\newcommand{\Hc}{\mathcal H^{(n-1)/2}_{1}}
\newcommand{\Cspace}{\mathcal H\mathcal C}
\newcommand{\Aspace}{\mathcal A}
\newcommand{\Xspace}{\mathcal X}

\begin{document}
\arraycolsep=1pt

\title[Hausdorff--Choquet Angular Spaces]
{Hausdorff--Choquet Angular Spaces and Weak-Type $(1,1)$ Bounds for Rough Maximal Operators}

\author{Yanping Chen$^\ast$}
\address{Department of Mathematics,
	Northeastern University;
	State Key Laboratory of Synthetical Automation for Process Industries,
	Shenyang 110004, China}
\email{yanpingch@126.com}

\author{Zhengyang Ji}
\address{Department of Mathematics,
	Northeastern University,
	Shenyang 110004, China}
\email{zhengyang.ji.zjut@gmail.com}

\thanks{Yanping Chen is the corresponding author. The research was partially supported by the National Natural Science Foundation of China ( Grant No.  12525105 and 12371092).
\textcolor{red}{\today}/Final version.}

 
\begin{abstract}
	In the present paper, we consider the maximal operator
	\[
	\mathcal M_{\Omega}f(x)
	:=
	\sup_{r>0}\frac{1}{r^n}
	\int_{|y|<r}
	|f(x-y)|
	\left|
	\Omega\!\left(\frac{y}{|y|}\right)
	\right|\,dy.
	\]
	A longstanding open conjecture raised by E.~M.~Stein asks whether
	the maximal operator $\mathcal M_{\Omega}$ is of weak type $(1,1)$
	when $\Omega$ is merely in $L^1(\mathbb S^{n-1})$. We partially settle this problem by proving weak type $(1,1)$ bounds of
	$\mathcal M_{\Omega}$ with kernel
	$\Omega\in\mathcal X(\mathbb S^{n-1})$,
	yielding a significant improvement over the work of M.~Christ and Rubio de Francia.
	Here
	$\mathcal X(\mathbb S^{n-1})$ is the space related to the Hausdorff--Choquet angular space and $$
	L\log^+\!L(\mathbb S^{n-1})
	\subsetneq
	\mathcal X(\mathbb S^{n-1})\subset L^1(\mathbb S^{n-1}).
	$$
	Prior to our work, the condition $\Omega \in L\log^+\!L(\mathbb S^{n-1}) $ remained the best known size restriction
	guaranteeing the weak-type$ (1,1)$ bounds for all dimensions $n > 2$. To our knowledge, the present paper is the first to relax this threshold for $n>2$.

	Finally, for the general Hausdorff--Choquet scale
	$\Cspace_{\alpha}$, we show that
	$\alpha=(n-1)/2$ is the sharp exponent for uniform weak type $(1,1)$
	estimates.

\end{abstract}

\maketitle

\section{Introduction}\label{sec:intro}

In recent decades, a broad class of operators related to the Calderón–Zygmund singular integrals, yet lacking the smoothness assumptions imposed in the classical theory, has attracted substantial research attention. Investigations into the boundedness of rough singular integral operators date back to the seminal work of Calder\'on and Zygmund \cite{CalderonZygmund1956}, in which the authors first introduced singular integral operators with homogeneous kernels.

Let $\R^n$ be the $n$-dimensional Euclidean space of dimension $n\ge2$ and
\(
 \Sph=\{x\in\R^n:|x|=1\},
\)
the unit sphere in $\R^n$. Let $\sig$ denote the normalized surface measure on $\Sph$. Let $\Omega$ be a measurable function on $\Sph$ and be homogeneous of degree $0$ on $\Sph$. 
Consider the maximal function
\begin{equation}\label{eq:rough-maximal}
 \mathcal M_{\Omega}f(x)
 :=
 \sup_{r>0}r^{-n}
 \int_{|y|<r}|f(x-y)|\,|\Omega(y)|\dd y,
\end{equation}
for $f$ bounded and measurable with compact support in $\R^n$.

It is a consequence of the method of rotations of Calder\'on and Zygmund \cite{CalderonZygmund1956} that $\mathcal M_{\Omega}$ extends to an operator bounded on $L^p(\R^n)$ for all $p>1$, provided $\Omega\in L^1(\Sph)$. $\mathcal M_{\Omega}$ is a variant of the Hardy--Littlewood maximal function, so one might hope that it is of weak type $(1,1)$. Results of this type cannot be obtained by the method of rotations, since weak $L^1$ fails to be a normed linear space. In \cite{Fefferman1978}, R.~Fefferman proved that \(\mathcal M_\Omega\) is of weak type
\((1,1)\) whenever the kernel \(\Omega\) has finite
\(L^1\)-entropy. He also showed that the \(L^1\)-Dini condition
implies finite \(L^1\)-entropy; see also
\cite[Chapter~II, Section~5.19, p.~83]{Stein1993}.
 In a paper of S. Hudson \cite{Hudson1987}, who, in the case $n=2$, showed $\mathcal M_{\Omega}$ to be of weak type $(1,1)$ under the minimal size condition $\Omega\in L^1(\mathbb S^1)$, but assuming that $\Omega$ is monotone and assuming an extra technical hypothesis on $\Omega$. In \cite{Christ1988}, M. Christ used the $T T^*$ method to show that $\mathcal M_{\Omega}$ is of weak type $(1,1)$ in $\R^2$ whenever $\Omega\in L^q(\mathbb S^1)$ for some $q>1$. In 1988, M. Christ and Rubio de Francia \cite{ChristRubio1988} applied a variant of the geometric method and the $T T^*$ method to prove the weak type $(1,1)$ bounds for the maximal operator $\mathcal M_{\Omega}$ whenever $\Omega\in L\log^+\! L(\Sph)$. Now let us summarize the dominant result as follows.

\medskip
\noindent{\bf Theorem A (Christ and Rubio de Francia \cite{ChristRubio1988}).}
Let $\Omega\in L\log^+\! L(\Sph)$. Then $\mathcal M_{\Omega}$ is of weak type $(1,1)$, that is, for any $\lambda>0$,
\[
 \lambda\,|\{x\in\R^n:\mathcal M_{\Omega}f(x)>\lambda\}|
 \lesssim_n
 \bigl(1+\|\Omega\|_{L\log^+\! L})\bigr)\|f\|_1,
\]
where
\[
 \|\Omega\|_{L\log^+\! L}
 :=
 \int_{\Sph}|\Omega(\theta)|\log\!\bigl(2+|\Omega(\theta)|\bigr)\dd\sig(\theta).
\]

Since then, maximal operators and their endpoint estimates have continued to play an important role in harmonic analysis. For related developments concerning rough kernels, lifted maximal operators, and maximal operators on trees, see \cite{DLYYZ,FS,HuLaiTaoXue2025,Lai2024Noncommutative,Lai2025MaximalTruncated,QinWuXueYangZhang2026,SeegerTao2001,SjogrenSoria1996,S1,YYZ,YYZ1} and the references therein. 

M.~Christ and Rubio de Francia \cite{ChristRubio1988} specially mentioned that it is not
known whether the conclusion is valid for all $\Omega\in L^1(\Sph)$. In particular,
E.~M.~Stein explicitly raised this question in his book \cite{Stein1993}. \begin{question}\label{q:Stein-weak-type}
	 Whether $M_\Omega$ is of weak type $(1,1)$ when  $\Omega$ is merely in $L^1(\Bbb S^{n-1})$?
\end{question}This long-standing open problem has been repeatedly emphasized by E. M. Stein in another work \cite[p1139]{St1}, which is also addressed in Grafakos' textbook \cite{Grafakos2014Classical}, the 1999 survey by Grafakos and Stefanov \cite[Question~8]{GrafakosStefanov1999}, and the 2001 monograph written by Duoandikoetxea \cite[p88]{Duoandikoetxea2001}.

This long-standing endpoint problem has remained open for more than
three decades and the full
$L^1(\Sph)$ problem still open.

In this paper, we give a partial answer to this open problem. More precisely, we construct a kernel space $\Xspace(\Sph)$ such that
\begin{equation}\label{eq:Y-goal-intro}
 L\log^+\! L(\Sph)\subsetneq\Xspace(\Sph)\subset L^1(\Sph),
\end{equation}
and prove that $\mathcal M_{\Omega}$ is of weak type $(1,1)$ for every $\Omega\in\Xspace(\Sph)$.

The construction of $\mathcal X(\mathbb S^{n-1})$ is motivated by the desire to measure not only  the size of the kernel, but also its
geometric concentration on the sphere. This naturally leads to Hausdorff contents and  Choquet integrals, which play a
well-established role in potential theory and harmonic analysis.
Adams investigated Choquet integration with respect to Hausdorff capacity and derived maximal-function estimates within this capacitary framework
\cite{Adams1988Choquet}; see also his later survey
\cite{Adams1998Choquet}. Later, Orobitg and Verdera established a new family  of inequalities relating Hausdorff content, Choquet integrals, and the Hardy–Littlewood maximal operator
\cite{OrobitgVerdera1998}. More recently,  the theory of capacitary maximal-function
 has been extended to wider classes of Hausdorff contents
and outer capacities \cite{BasakChenRoychowdhurySpector2025}.

By contrast, Hausdorff–Choquet quantities assume a different role in the present paper. Rather than using Hausdorff content as an ambient capacity
on $\mathbb R^n$, we use the associated Choquet integral to measure
the geometric concentration of the kernel on $\mathbb S^{n-1}$.

We now introduce the definitions of some spaces. For $0<\alpha\le n-1$ and $E\subset\Sph$, define the unit-scale Hausdorff content by
\[
 \mathcal H_1^\alpha(E)
 :=
 \inf\left\{
 \sum_j r_j^\alpha:
 E\subset\bigcup_jB(\theta_j,r_j),\quad 0<r_j\le1
 \right\},
\]
where $B(\theta,r)$ denotes a geodesic ball. For a nonnegative measurable function $u$ on $\Sph$, define the associated Choquet integral by
\[
 \int_{\Sph}u\dd\mathcal H_1^\alpha
 :=
 \int_0^\infty\mathcal H_1^\alpha(\{u>t\})\dd t.
\]
For $\alpha=(n-1)/2$, set
\begin{equation}\label{eq:rhoC-intro}
 \rho_{\Cspace}(\Omega)
 :=
 \inf_{\widetilde\Omega=\Omega\ {\rm a.e.}\,[\sig]}
 \int_{\Sph}|\widetilde\Omega|^{1/2}\dd\mathcal H_1^{(n-1)/2},
 \qquad
 \|\Omega\|_{\Cspace}:=\rho_{\Cspace}(\Omega)^2,
\end{equation}
and define
\[
 \Cspace(\Sph):=\{\Omega:\rho_{\Cspace}(\Omega)<\infty\}.
\]
We call $\Cspace(\Sph)$ the \emph{Hausdorff--Choquet angular space}.

Our first result gives an atomic description of this space. For a geodesic ball $B\subset\Sph$, write
\[
 u_B:=\frac{\1_B}{\sig(B)}.
\]

\begin{definition}[Geodesic-ball atomic space]\label{def:atomic}
A measurable function $a$ is called a geodesic-ball atom if there exists a geodesic ball $B\subset\Sph$ such that
\[
 \abs a\le u_B.
\]
Define the geodesic-ball atomic space $\Aspace(\Sph)$  as follows.
\begin{equation}\label{eq:atomic-majorant-norm}
 \norm{\Omega}_{\Aspace}
 :=
 \inf
 \left\{
 \left(\sum_j\sqrt{c_j}\right)^2:
 \abs\Omega\le\sum_jc_ju_{B_j}\ {\rm a.e.},\
 c_j\ge0
 \right\}.
\end{equation}
\end{definition}
 
\begin{theorem}\label{thm:atomic-main}
The spaces $\Cspace(\Sph)$ and $\Aspace(\Sph)$ coincide with equivalent quasi-norms; namely,
\[
 \|\Omega\|_{\Aspace}\asymp_n\|\Omega\|_{\Cspace}.
\]
Moreover, $\Cspace(\Sph)$ is a complete quasi-Banach space, and $\|\cdot\|_{\Aspace}$ is an equivalent $1/2$-norm.
\end{theorem}

To place our new space within the context of classical  kernel spaces, we further prove the following.
\begin{theorem}\label{thm:comparison-main}
There are continuous embeddings
\[
 L^\infty(\Sph)\subset\Cspace(\Sph)
 \subsetneq L^{1,1/2}(\Sph)\subset L^1(\Sph).
\]
Moreover, there exists a nonnegative kernel $\Omega_*\in\Cspace(\Sph)$ such that
\[
 \Omega_*\notin L\log^+\! L(\Sph),
 \qquad
 \Omega_*\notin L^p(\Sph)
 \quad\text{for every }p>1.
\]
\end{theorem}

Our  main  results  can be stated as follows.

\begin{theorem}\label{thm:main}
There exists a constant $C_n$, depending only on $n$, such that for every $\Omega\in\Cspace(\Sph)$, every $f\in L^1(\R^n)$, and every $\lambda>0$,
\begin{equation}\label{eq:main-weak-intro}
 |\{x\in\R^n:\mathcal M_{\Omega}f(x)>\lambda\}|
 \le
 \frac{C_n}{\lambda}\|\Omega\|_{\Cspace}\|f\|_1.
\end{equation}
\end{theorem}

The second assertion of Theorem \ref{thm:comparison-main} shows that the Hausdorff--Choquet condition admits kernels beyond the classical $L\log^+\! L$ class. It is important, however, to distinguish this statement from the stronger inclusion $L\log^+\! L(\Sph)\subset\Cspace(\Sph)$, which is not asserted here. To construct  a space that contains the  entire $L\log^+\! L$ space, define
\[
 \Xspace(\Sph):=L\log^+\! L(\Sph)+\Cspace(\Sph).
\]
For later use, we employ the homogeneous Luxemburg norm
associated with the Young function
$\Phi(t)=t\log(2+t)$; see, for example,
\cite[Chapter~IV, Section~8]{BennettSharpley1988}:
\begin{equation}\label{eq:LlogL-lux}
	\|\Omega\|_{L\log^+L}
	:=
	\inf\left\{
	a>0:
	\int_{\Sph}
	\frac{|\Omega(\theta)|}{a}
	\log\!\left(2+\frac{|\Omega(\theta)|}{a}\right)
	\,d\sigma(\theta)\le1
	\right\}.
\end{equation}
and define
\begin{equation}\label{eq:X-control-main}
 \mathcal N_{\Xspace}(\Omega)
 :=
 \inf_{\substack{\Omega=\Omega_1+\Omega_2\\
                  \Omega_1\in L\log^+\! L(\Sph)\\
                  \Omega_2\in\Cspace(\Sph)}}
 \left(
 \|\Omega_1\|_{L\log^+\! L}+\|\Omega_2\|_{\Cspace}
 \right).
\end{equation}
The classical $L\log^+\! L$ estimate and Theorem \ref{thm:main} imply the following.

\begin{theorem}\label{thm:bridge}
There exists a constant $C_n$, depending only on $n$, such that for every $\Omega\in\Xspace(\Sph)$, every $f\in L^1(\R^n)$, and every $\lambda>0$,
\begin{equation}\label{eq:X-weak-main}
 |\{x\in\R^n:\mathcal M_{\Omega}f(x)>\lambda\}|
 \le
 \frac{C_n}{\lambda}\,
 \mathcal N_{\Xspace}(\Omega)\|f\|_1.
\end{equation}
Moreover,
\[
 L\log^+\! L(\Sph)\subsetneq\Xspace(\Sph)\subset L^1(\Sph).
\]
\end{theorem}

Theorem \ref{thm:bridge} gives a partial answer to the question raised by E.M. Stein (see also Question \ref{q:Stein-weak-type}): the classical $L\log^+\! L$ sufficient condition is strictly widened, while the resulting kernel space remains contained in $L^1(\Sph)$.

Finally, we consider the more general Hausdorff--Choquet scale. For $0<\alpha\le n-1$, define
\[
 \rho_{\Cspace_\alpha}(\Omega)
 :=
 \inf_{\omega=\Omega\ {\rm a.e.}\,[\sig]}
 \int_{\Sph}|\omega|^{1/2}\dd\mathcal H_1^\alpha,
 \qquad
 \|\Omega\|_{\Cspace_\alpha}:=\rho_{\Cspace_\alpha}(\Omega)^2.
\]
Then $\Cspace_{(n-1)/2}=\Cspace$, and the exponent $(n-1)/2$ is sharp.

\begin{theorem}\label{thm:alpha-main}
Let $0<\alpha\le n-1$. A uniform estimate
\[
 |\{x\in\R^n:\mathcal M_{\Omega}f(x)>\lambda\}|
 \le
 \frac{C_n}{\lambda}\|\Omega\|_{\Cspace_\alpha}\|f\|_1
\]
for all nonnegative measurable $\Omega$, all $f\in L^1(\R^n)$, and all $\lambda>0$ holds if and only if
\[
 0<\alpha\le\frac{n-1}{2}.
\]
In particular, the sharp exponent is
\[
 \alpha_c=\frac{n-1}{2}.
\]
\end{theorem}

The paper is organized as follows: In Section \ref{sec:prelim}, we give the required preliminaries. In Section \ref{sec:HC-structure}, we introduce the Hausdorff--Choquet angular space, establish its atomic characterization and investigate its comparison properties, and provide the proofs of Theorems \ref{thm:atomic-main}-\ref{thm:comparison-main}. In Section \ref{sec:weak-main}, we establish the single-geodesic-ball estimate, the weak type $(1,1)$ bounds, the sum-space theorem, and the sharp-exponent theorem and give the proofs of Theorems \ref{thm:main}-\ref{thm:alpha-main}.

\medskip
\noindent{\bf Notation.}\begin{enumerate}
\item Throughout the paper, $n\ge2$ is fixed. The symbol $\Sph$ denotes the unit sphere, equipped with the geodesic distance
\(
d_{\Sph}(\theta,\varphi)
:=
\arccos\langle\theta,\varphi\rangle
\in[0,\pi].
\)

\item
For $\theta\in\Sph$ and $0<r\leq1$, let
\(
B(\theta,r)
:=
\bigl\{
\varphi\in\Sph:
d_{\Sph}(\theta,\varphi)<r
\bigr\}
\)
be the open geodesic ball centered at $\theta$ with geodesic radius $r$; write $B_{\R^n}(\theta,r):=
\bigl\{
\varphi\in\Sph:
d_{\R^n}(\theta,\varphi)<r
\bigr\}$ for the open Euclidean ball centered at $\theta$ with radius $r$.
\item The  $\sigma$  denotes normalized surface measure on $\Sph$.
For a measurable set $E$, $\1_E$ denotes its indicator function.

\item For $A,B\ge0$, the notation $A\lesssim_nB$ means that $A\le CB$ for some constant $C>0$ depending only on $n$; $A\asymp_nB$ means that both $A\lesssim_nB$ and $B\lesssim_nA$ hold.

\item For open sets $U$ and $\widetilde U$, the notation $U\Subset\widetilde U$ means that $U$ is compactly contained in $\widetilde U$, that is, $\overline U$ is compact and $\overline U\subset\widetilde U$.

\item For weak $L^1$, we use the quasi-norm
\(
\norm F_{L^{1,\infty}(\R^n)}
:=
\sup_{\lambda>0}
\lambda\,\abs{\{x\in\R^n:\abs{F(x)}>\lambda\}}.
\)
The abbreviation a.e. always refers to the measure specified in the relevant context.\end{enumerate}

\section{Preliminaries}\label{sec:prelim}

\subsection{Spherical Hausdorff Content and the Choquet Integral}

\begin{lemma}\label{lem:ahlfors}
There exist constants $c_n,C_n>0$, depending only on $n$, such that for every
$\theta\in\Sph$ and $0<r\le1$,
\begin{equation}\label{eq:ball-measure}
 c_nr^{n-1}\le\sig(B(\theta,r))\le C_nr^{n-1}.
\end{equation}
\end{lemma}

\begin{proof}
By rotational invariance of the unit sphere, it suffices to consider a fixed center. The measure of a geodesic ball satisfies
\[
 \sig(B(\theta,r))
 =c_n\int_0^r(\sin t)^{n-2}\dd t,
 \qquad 0<r\le1,
\]
where $c_n>0$ depends only on $n$. Since $\sin t\asymp t$ for $0\le t\le1$,
\[
 \sig(B(\theta,r))
 \asymp_n
 \int_0^r t^{n-2}\dd t
 \asymp_n r^{n-1}.
\]
\end{proof}

\begin{definition}[Unit-scale Hausdorff content]\label{def:hausdorff-content}
For $0<\alpha\le n-1$ and $E\subset\Sph$, following the restricted-scale Hausdorff content in \cite[Section 2]{Papasoglu2020}, define
\[
 \mathcal H_1^\alpha(E)
 :=\inf\left\{
 \sum_jr_j^\alpha:
 E\subset\bigcup_jB(\theta_j,r_j),\quad 0<r_j\le1
 \right\}.
\]
Throughout the paper we set $\alpha=(n-1)/2$, except in Subsection \ref{subsec:critical-alpha}, where a general $\alpha$ is treated explicitly.
\end{definition}

\begin{proposition}[Two basic estimates]\label{prop:content-estimates}
For every $\theta\in\Sph$, $0<r\le1$, and measurable set $E\subset\Sph$, one has
\begin{equation}\label{eq:ball-content}
 \Hc(B(\theta,r))\asymp_n r^{(n-1)/2},
 \qquad
 \Hc(E)\gtrsim_n\sig(E)^{1/2}.
\end{equation}
\end{proposition}

\begin{proof}
	The geodesic ball itself is an admissible covering; hence
	\(
	\mathcal H^{(n-1)/2}_1(B(\theta,r))
	\le r^{(n-1)/2}.
	\)
	
	Conversely, suppose that
	\[
	B(\theta,r)\subset \bigcup_j B(\theta_j,r_j),
	\qquad 0<r_j\le 1.
	\]
	
	By the subadditivity of surface measure and Ahlfors regularity,
	\[
	r^{n-1}
	\lesssim_n
	\sigma(B(\theta,r))
	\le
	\sum_j \sigma(B(\theta_j,r_j))
	\lesssim_n
	\sum_j r_j^{n-1}.
	\]
	Moreover,
	\(
	\sum_j r_j^{n-1}
	\le
	\left(\sum_j r_j^{(n-1)/2}\right)^2,
	\)
	and therefore
	\(
	\sum_j r_j^{(n-1)/2}\gtrsim_n r^{(n-1)/2}.
	\)
	Taking the infimum over all admissible coverings gives
	\(
	\mathcal H^{(n-1)/2}_1(B(\theta,r))
	\gtrsim_n r^{(n-1)/2}.
	\)
	Together with the upper bound above, this yields
	\(
	\mathcal H^{(n-1)/2}_1(B(\theta,r))
	\asymp_n r^{(n-1)/2}.
	\)
	
	Finally, suppose that
	\[
	E\subset \bigcup_j B(\theta_j,r_j),
	\qquad 0<r_j\le1.
	\]
	Again, by the subadditivity of measure and Ahlfors regularity,
	\[
	\sigma(E)
	\le
	\sum_j\sigma(B(\theta_j,r_j))
	\lesssim_n
	\sum_j r_j^{n-1}
	\le
	\left(\sum_jr_j^{(n-1)/2}\right)^2.
	\]
	Therefore,
	\(
	\sum_jr_j^{(n-1)/2}
	\gtrsim_n
	\sigma(E)^{1/2}.
	\)
	Taking the infimum over all admissible coverings, we obtain
	\[
	\mathcal H^{(n-1)/2}_1(E)
	\gtrsim_n
	\sigma(E)^{1/2}.
	\]
\end{proof}

\begin{definition}[Choquet integral {\cite{OrobitgVerdera1998}}]
For a nonnegative measurable function $u$, define
\[
 \int_{\Sph}u\dd\Hc
 :=\int_0^{\infty}\Hc(\{u>t\})\dd t.
\]
\end{definition}

For $E\subset\Sph$, let $\mathcal H_{\infty,\R^n}^\alpha(E)$ denote the Euclidean Hausdorff content, where $E$ is a subset of $\R^n$; namely,
\[
 \mathcal H_{\infty,\R^n}^\alpha(E)
 :=
 \inf\left\{
 \sum_i r_i^\alpha:
 E\subset\bigcup_iB_{\R^n}(x_i,r_i),\quad r_i>0
 \right\}.
\]

\begin{lemma}[Quasi-subadditivity of the Choquet integral]\label{lem:choquet-subadd}
There exists a constant $C_n$, depending only on $n$, such that for every sequence of nonnegative measurable functions $\{u_j\}_{j\ge1}$,
\begin{equation}\label{eq:choquet-subadd}
 \int_{\Sph}\sum_{j=1}^{\infty}u_j\dd\Hc
 \le C_n\sum_{j=1}^{\infty}\int_{\Sph}u_j\dd\Hc.
\end{equation}
\end{lemma}

\begin{proof}
Set $\alpha=(n-1)/2$. The geodesic distance on $\Sph$ and the Euclidean distance satisfy
\[
 \abs{\theta-\varphi}
 \le d_{\Sph}(\theta,\varphi)
 \le \frac{\pi}{2}\abs{\theta-\varphi},
 \qquad \theta,\varphi\in\Sph.
\]
Since every geodesic ball $B(\theta,r)$ is contained in the Euclidean ball $B_{\R^n}(\theta,r)$, we first have
\[
 \mathcal H_{\infty,\R^n}^\alpha(E)
 \le \mathcal H_1^\alpha(E).
\]
Conversely, take an arbitrary covering of $E$ by Euclidean balls,
\[
E\subset\bigcup_i B_{\R^n}(x_i,r_i),
\qquad r_i>0.
\]
If
\(
\sum_i r_i^\alpha=\infty,
\)
then there is nothing to prove. It therefore suffices to consider the case
\[
\sum_i r_i^\alpha<\infty
\]
Remove all balls disjoint from  $E$. For each remaining ball $B_{\R^n}(x_i,r_i)$, choose
\[
\theta_i\in E\cap B_{\R^n}(x_i,r_i).
\]

First consider the case $r_i\le1/\pi$. If
\[
\varphi\in E\cap B_{\R^n}(x_i,r_i),
\]
then $\theta_i,\varphi\in B_{\R^n}(x_i,r_i)$, and the Euclidean triangle inequality gives
\[
\abs{\theta_i-\varphi}
\le
\abs{\theta_i-x_i}+\abs{x_i-\varphi}
<2r_i.
\]
Using the comparison between the geodesic distance on $\Sph$ and the chordal distance,
\[
d_{\Sph}(\theta_i,\varphi)
\le
\frac{\pi}{2}\abs{\theta_i-\varphi}
<\pi r_i.
\]
Therefore,
\[
E\cap B_{\R^n}(x_i,r_i)
\subset
B(\theta_i,\pi r_i).
\]
Since \(r_i\le 1/\pi\), we have \(\pi r_i\le 1\). Thus the geodesic ball on the right-hand side is admissible in the definition of \(\mathcal H_1^\alpha\), and its contribution to the covering sum is
\[
(\pi r_i)^\alpha
=
\pi^\alpha r_i^\alpha.
\]

Next consider the case $r_i>1/\pi$. Choose a maximal $1/2$-separated set
\[
\{\eta_\nu\}_{\nu=1}^{N_n}\subset\Sph.
\]
where $1/2$-separated means that
\[
d_{\Sph}(\eta_\nu,\eta_\mu)\ge\frac12,
\qquad \nu\ne\mu.
\]
By the maximality of this set,
\[
\Sph
\subset
\bigcup_{\nu=1}^{N_n}
B(\eta_\nu,1/2).
\]
On the other hand, the geodesic balls
\[
B(\eta_\nu,1/4),
\qquad 1\le\nu\le N_n,
\]
are pairwise disjoint. By Lemma~\ref{lem:ahlfors},
\[
\begin{aligned}
	1
	=\sig(\Sph)
	&\ge
	\sum_{\nu=1}^{N_n}
	\sig\bigl(B(\eta_\nu,1/4)\bigr)\\
	&\gtrsim_n N_n.
\end{aligned}
\]
Hence
\[
N_n\lesssim_n1.
\]
Thus, the geodesic balls of radius $1/2$ constructed above form a finite covering of the unit sphere, whose cardinality depends only on the dimension $n$.

In particular, when $r_i>1/\pi$, the set $E\cap B_{\R^n}(x_i,r_i)$ can be covered by the $N_n$ geodesic balls of radius $1/2$ constructed above. The sum of the $\alpha$-powers of their radii is
\[
N_n\left(\frac12\right)^\alpha.
\]
Since $r_i>1/\pi$,
\[
r_i^\alpha>\pi^{-\alpha},
\]
and hence
\[
N_n\left(\frac12\right)^\alpha
\le
N_n\left(\frac{\pi}{2}\right)^\alpha r_i^\alpha
\lesssim_n r_i^\alpha,
\]
where in the last step we used
\(
\alpha=(n-1)/2.
\)

Let
\[
I_{\mathrm{small}}
:=
\{i:r_i\le1/\pi\},
\qquad
I_{\mathrm{large}}
:=
\{i:r_i>1/\pi\}.
\]
Combining the two covering procedures above,
\[
E
\subset
\bigcup_{i\in I_{\mathrm{small}}}
B(\theta_i,\pi r_i)
\,\cup\,
\bigcup_{i\in I_{\mathrm{large}}}
\bigcup_{\nu=1}^{N_n}
B(\eta_\nu,1/2).
\]
This is an admissible covering for $\mathcal H_1^\alpha(E)$, and the sum of the $\alpha$-powers of its radii satisfies
\[
\begin{aligned}
	\mathcal H_1^\alpha(E)
	&\le
	\sum_{i\in I_{\mathrm{small}}}
	(\pi r_i)^\alpha
	+
	\sum_{i\in I_{\mathrm{large}}}
	N_n\left(\frac12\right)^\alpha\\
	&\lesssim_n
	\sum_{i\in I_{\mathrm{small}}}r_i^\alpha
	+
	\sum_{i\in I_{\mathrm{large}}}r_i^\alpha\\
	&=
	\sum_i C_n r_i^\alpha.
\end{aligned}
\]
Finally, taking the infimum over all Euclidean-ball coverings of $E$, we obtain
\[
\mathcal H_1^\alpha(E)
\lesssim_n
\mathcal H_{\infty,\R^n}^\alpha(E).
\]
Thus, for every $E\subset\Sph$,
\begin{equation}\label{eq:spherical-euclidean-content}
 \mathcal H_1^\alpha(E)
 \asymp_n
 \mathcal H_{\infty,\R^n}^\alpha(E).
\end{equation}
By the level-set definition of the Choquet integral, \eqref{eq:spherical-euclidean-content} further implies that, for every nonnegative function $w$,
\begin{equation}\label{eq:spherical-euclidean-choquet}
 \int_{\Sph}w\dd\mathcal H_1^\alpha
 \asymp_n
 \int_{\Sph}w\dd\mathcal H_{\infty,\R^n}^\alpha.
\end{equation}
On the other hand, the countable quasi-subadditivity of the
Choquet integral with respect to  Euclidean Hausdorff content is
established in \cite[(2.4)]{ChenOoiSpector2024}.
Applying this result to $\Sph\subset\R^n$ with $\alpha=(n-1)/2$, and then using \eqref{eq:spherical-euclidean-choquet}, we obtain
\begin{align*}
 \int_{\Sph}\sum_{j=1}^{\infty}u_j\dd\mathcal H_1^\alpha
 &\lesssim_n
 \int_{\Sph}\sum_{j=1}^{\infty}u_j
 \dd\mathcal H_{\infty,\R^n}^\alpha\\
 &\lesssim_n
 \sum_{j=1}^{\infty}
 \int_{\Sph}u_j\dd\mathcal H_{\infty,\R^n}^\alpha\\
 &\lesssim_n
 \sum_{j=1}^{\infty}
 \int_{\Sph}u_j\dd\mathcal H_1^\alpha.
\end{align*}
This is \eqref{eq:choquet-subadd}.
\end{proof}

\subsection{Preliminaries on Maximal Operators and Classical Kernel Spaces}

We collect here the classical results that will be used directly in what follows. First, let $ \mathcal{M}$ denote the Hardy--Littlewood maximal operator, defined by
\[
 \mathcal{M} F(x)
 :=
 \sup_{r>0}
 \frac1{\abs{B_{\R^n}(0,r)}}
 \int_{B_{\R^n}(0,r)}\abs{F(x-y)}\dd y.
\]

\begin{lemma}[Hardy--Littlewood maximal theorem
{\cite[Chapter~I]{Stein1970}}]
\label{lem:HL}
There exists a constant $C_n$, depending only on $n$, such that for every
$F\in L^1(\R^n)$ and $\lambda>0$,
\[
 \abs{\{x:\mathcal{M}F(x)>\lambda\}}
 \le
 \frac{C_n}{\lambda}\norm F_1.
\]
Equivalently,
\[
 \norm{\mathcal{M}F}_{L^{1,\infty}(\R^n)}
 \lesssim_n
 \norm F_1.
\]
\end{lemma}

\begin{lemma}[Weak type control for averages over a fixed shape]\label{lem:fixed-shape}
Let $Q_0\subset\R^n$ be a bounded measurable set satisfying
$0<\abs{Q_0}<\infty$ and $Q_0\subset B_{\R^n}(0,C_0)$. Define
\[
 \mathcal N_{Q_0}F(x)
 :=\sup_{R>0}\frac1{\abs{RQ_0}}
 \int_{RQ_0}\abs{F(x-y)}\dd y.
\]
Then there exists a constant $C_{n,Q_0}$, depending only on $n$ and $Q_0$, such that
\[
 \abs{\{x:\mathcal N_{Q_0}F(x)>\lambda\}}
 \le\frac{C_{n,Q_0}}{\lambda}\norm F_1.
\]
\end{lemma}

\begin{proof}
Since $RQ_0\subset B_{\R^n}(0,C_0R)$ and $\abs{RQ_0}=R^n\abs{Q_0}$,
\[
 \frac1{\abs{RQ_0}}\int_{RQ_0}\abs{F(x-y)}\dd y
 \le C_{n,Q_0}\mathcal {M} F(x).
\]
Taking the supremum over $R>0$ and then applying Lemma \ref{lem:HL} yields the desired result.
\end{proof}

\begin{lemma}[The $L\log^+\! L$ weak-type estimate of Christ and Rubio de Francia
{\cite{ChristRubio1988}}]
\label{lem:CR-modular}
Let
\[
 \Phi(t)=t\log(2+t),\qquad t\ge0.
\]
If $\Omega\in L\log^+ \!L(\Sph)$, then
\begin{equation}\label{eq:CR-modular-form}
 \norm{\mathcal M_{\Omega}f}_{L^{1,\infty}}
 \lesssim_n
 \left(
 1+\int_{\Sph}\Phi(\abs\Omega)\dd\sig
 \right)\norm f_1.
\end{equation}
\end{lemma}

\begin{lemma}
\label{lem:classical-llogl}
There exists a constant $C_n$, depending only on $n$, such that for every
$\Omega\in L\log^+\! L(\Sph)$, $f\in L^1(\R^n)$, and $\lambda>0$,
\begin{equation}\label{eq:classical-llogl}
 \abs{\{x:\mathcal M_{\Omega}f(x)>\lambda\}}
 \le
 \frac{C_n}{\lambda}
 \norm{\Omega}_{L\log^+\! L}\norm f_1.
\end{equation}
\end{lemma}

\begin{proof}
If $\Omega=0$, the conclusion is immediate. Suppose that $\Omega\ne0$, and take any
\[
 a>\norm{\Omega}_{L\log^+\! L}.
\]
By the definition of the norm in \eqref{eq:LlogL-lux},
\[
 \int_{\Sph}
 \Phi\!\left(\frac{\abs\Omega}{a}\right)\dd\sig
 \le1.
\]
Set $\widetilde\Omega:=\Omega/a$. Applying Lemma \ref{lem:CR-modular} to $\widetilde\Omega$ gives
\[
 \norm{\mathcal M_{\widetilde\Omega}f}_{L^{1,\infty}}
 \lesssim_n\norm f_1.
\]
On the other hand, by the homogeneity of \eqref{eq:rough-maximal},
\[
 \mathcal M_{\Omega}f=a\mathcal M_{\widetilde\Omega}f.
\]
Hence
\[
 \norm{\mathcal M_{\Omega}f}_{L^{1,\infty}}
 \lesssim_n a\norm f_1.
\]
Letting $a\downarrow\norm{\Omega}_{L\log^+ \!L}$ yields \eqref{eq:classical-llogl}.
\end{proof}

\begin{remark}\label{rem:modular-vs-norm}
The integral quantity
\[
 \int_{\Sph}\abs\Omega\log(2+\abs\Omega)\dd\sig
\]
characterizes the same $L\log^+\! L$ space, but is not itself homogeneous in $\Omega$.
Through Luxemburg normalization, Lemma \ref{lem:classical-llogl} rewrites the classical estimate in a homogeneous form fully compatible with
\[
 \mathcal M_{c\Omega}=\abs c\,\mathcal M_{\Omega}
\].
\end{remark}

\begin{lemma}[{\cite[~p48]{Grafakos2014Classical}}]
\label{lem:lorentz-L1}
Let $(Y,\mu)$ be a finite measure space, and define
\[
 \norm h_{L^{1,1/2}(Y)}
 :=
 \left(
 \int_0^\infty
 t^{-1/2}\mu(\{\abs h>t\})^{1/2}\dd t
 \right)^2.
\]
Then
\[
 L^{1,1/2}(Y)\subset L^1(Y)
\]
continuously. In particular, when $Y=\Sph$ and $\mu=\sig$,
\[
 \norm h_{L^1(\Sph)}
 \lesssim
 \norm h_{L^{1,1/2}(\Sph)}.
\]
\end{lemma}

\begin{proof}
Let $F_k:=\{\abs h>2^k\}$. By the layer-cake decomposition,
\[
 \norm h_{L^1}
 \le
 \sum_{k\in\mathbb Z}2^{k+1}\mu(F_k).
\]
On the other hand,
\[
 \int_0^\infty
 t^{-1/2}\mu(\{\abs h>t\})^{1/2}\dd t
 \asymp
 \sum_{k\in\mathbb Z}
 2^{k/2}\mu(F_k)^{1/2}.
\]
Therefore,
\[
 \norm h_{L^1}
 \lesssim
 \sum_k
 \left(2^{k/2}\mu(F_k)^{1/2}\right)^2
 \le
 \left(
 \sum_k2^{k/2}\mu(F_k)^{1/2}
 \right)^2,
\]
which shows the assertion.
\end{proof}

\section{The Hausdorff--Choquet Angular Space and Its Structure}\label{sec:HC-structure}

\subsection{The Hausdorff--Choquet Space and Geodesic-Ball Atoms}\label{sec:atomic-space}

Throughout this subsection, we fix $\alpha=(n-1)/2$.

All function-space elements in this section are defined on equivalence classes modulo equality $\sigma$-almost everywhere. Let \(\omega\) be a measurable function representing such an equivalence class, and set
\[
 I(\omega)
 :=
 \int_{\Sph}\abs{\omega(\theta)}^{1/2}\dd\Hc(\theta).
\]

\begin{definition}[Hausdorff--Choquet space]\label{def:capacity}
Let \(\Omega\) denote an equivalence class of measurable functions modulo \(\sigma\)-almost-everywhere equality, and define
\begin{equation}\label{eq:rhoC}
 \rho_{\Cspace}(\Omega)
 :=
 \inf_{\omega=\Omega\ {\rm a.e.}\,[\sigma]} I(\omega),
 \qquad
 \norm{\Omega}_{\Cspace}:=\rho_{\Cspace}(\Omega)^2.
\end{equation}
Let
\[
 \Cspace(\Sph)
 :=
 \{\Omega:\rho_{\Cspace}(\Omega)<\infty\}.
\]
\end{definition}

\begin{definition}[General Hausdorff--Choquet scale]\label{def:C-alpha}
For $0<\alpha\le n-1$, define
\[
 \rho_{\Cspace_\alpha}(\Omega)
 :=\inf_{\omega=\Omega\ {\rm a.e.}\,[\sig]}
 \int_{\Sph}\abs{\omega}^{1/2}\dd\mathcal H_1^\alpha,
 \qquad
 \norm{\Omega}_{\Cspace_\alpha}
 :=\rho_{\Cspace_\alpha}(\Omega)^2.
\]
Then $\Cspace_{(n-1)/2}=\Cspace$, with identical quasi-norms.
\end{definition}

\begin{remark}
Taking the infimum over all $\sigma$-almost-everywhere representatives is essential. The operator \(\mathcal M_{\Omega}\) depends only on the equivalence class of \(\Omega\) modulo \(\sigma\)-almost-everywhere equality, whereas the Choquet integral may change under modifications on a \(\sigma\)-null set.
\end{remark}

\begin{proposition}\label{prop:layer}
For any measurable function $\omega$ on $\Sph$, let
\[
 F_k(\omega):=\{\theta\in\Sph:\abs{\omega(\theta)}>2^k\}.
\]
Then
\begin{align}
 I(\omega)
 &=\frac12\int_0^\infty t^{-1/2}
 \Hc(\{\abs\omega>t\})\dd t,
 \label{eq:continuous-layer}\\
 I(\omega)
 &\asymp
 \sum_{k\in\mathbb Z}2^{k/2}\Hc(F_k(\omega)).
 \label{eq:dyadic-layer}
\end{align}
\end{proposition}

\begin{proof}
By the definition of the Choquet integral,
\[
 I(\omega)
 =\int_0^\infty
 \Hc(\{\abs\omega^{1/2}>s\})\dd s.
\]
The change of variables $t=s^2$ yields \eqref{eq:continuous-layer}. Decompose $(0,\infty)$ dyadically as $(0,\infty)=\bigcup_{k\in\mathbb Z}[2^k,2^{k+1})$. If $2^k\le t<2^{k+1}$, then
\[
 F_{k+1}(\omega)
 \subset\{\abs\omega>t\}
 \subset F_k(\omega).
\]
Moreover,
\[
 \frac12\int_{2^k}^{2^{k+1}}t^{-1/2}\dd t
 =(\sqrt2-1)2^{k/2}.
\]
Summing over $k\in\mathbb Z$ and shifting the index gives \eqref{eq:dyadic-layer}.
\end{proof}

\begin{proposition}[Basic properties]\label{prop:basic-properties}
For $\Omega,\Omega_1,\Omega_2\in\Cspace(\Sph)$ and $c\ge0$, one has
\begin{align}
 \norm{c\Omega}_{\Cspace}
 &=c\norm{\Omega}_{\Cspace},
 \label{eq:C-homogeneous}\\
 \norm{\Omega_1+\Omega_2}_{\Cspace}^{1/2}
 &\lesssim_n
 \norm{\Omega_1}_{\Cspace}^{1/2}
 +\norm{\Omega_2}_{\Cspace}^{1/2},
 \label{eq:C-quasi-triangle}\\
 \norm{\Omega}_{L^1(\Sph)}
 &\lesssim_n\norm{\Omega}_{\Cspace}.
 \label{eq:C-L1}
\end{align}
Moreover, if $B\subset\Sph$ is a geodesic ball and $A\ge0$, then
\begin{equation}\label{eq:ball-spike}
 \norm{A\1_B}_{\Cspace}\asymp_n A\sig(B).
\end{equation}
\end{proposition}

\begin{proof}
Homogeneity follows directly from the definition. To prove the quasi-triangle inequality, fix $\varepsilon>0$ and choose measurable functions $\omega_i=\Omega_i$ $\sigma$-almost everywhere such that
\[
 I(\omega_i)\le\rho_{\Cspace}(\Omega_i)+\varepsilon,
 \qquad i=1,2.
\]
By
\[
 \abs{\omega_1+\omega_2}^{1/2}
 \le\abs{\omega_1}^{1/2}+\abs{\omega_2}^{1/2}
\]
and Lemma \ref{lem:choquet-subadd},
\[
 I(\omega_1+\omega_2)
 \lesssim_n I(\omega_1)+I(\omega_2).
\]
Letting $\varepsilon\downarrow0$ gives \eqref{eq:C-quasi-triangle}.

We next establish the $L^1$ estimate. For any measurable function $\omega=\Omega$ $\sigma$-almost everywhere, let $F_k=F_k(\omega)$. By
\[
 \abs\omega\le\sum_{k\in\mathbb Z}2^{k+1}\1_{F_k}
\]
and Tonelli's theorem,
\[
 \norm{\Omega}_{L^1}^{1/2}
 =\norm{\omega}_{L^1}^{1/2}
 \le\sum_{k\in\mathbb Z}2^{(k+1)/2}\sig(F_k)^{1/2}.
\]
By Proposition \ref{prop:content-estimates} and \eqref{eq:dyadic-layer},
\[
 \norm{\Omega}_{L^1}^{1/2}
 \lesssim_n I(\omega).
\]
Taking the infimum over all such functions $\omega$ and then squaring gives \eqref{eq:C-L1}.

Finally, taking $\omega=A\1_B$ in the definition, \eqref{eq:ball-content} and \eqref{eq:ball-measure} give
\[
 \norm{A\1_B}_{\Cspace}
 \le A\Hc(B)^2
 \lesssim_n A\sig(B).
\]
The reverse inequality follows from \eqref{eq:C-L1}:
\[
 A\sig(B)=\norm{A\1_B}_{L^1}
 \lesssim_n\norm{A\1_B}_{\Cspace}.
\]
\end{proof}

If $B\subset\Sph$ is a geodesic ball, set
\begin{equation}\label{eq:uB}
 u_B:=\frac{\1_B}{\sig(B)}.
\end{equation}

\begin{proposition}\label{prop:ball-majorant}
	For every $\Omega\in\Cspace(\Sph)$ and $\varepsilon>0$, there exist at most countably many geodesic balls $B_j\subset\Sph$ and coefficients $c_j\geq0$ such that
	\begin{equation}\label{eq:majorant-decomp}
		\abs{\Omega}
		\leq
		\sum_jc_ju_{B_j}
		\qquad \sigma\text{-a.e.},
	\end{equation}
	and
	\begin{equation}\label{eq:majorant-coefficient-sum}
		\sum_j\sqrt{c_j}
		\lesssim_n
		\rho_{\Cspace}(\Omega)+\varepsilon.
	\end{equation}
\end{proposition}

\begin{proof}
	By the definition of $\rho_{\Cspace}(\Omega)$, one may choose a measurable function $\omega$ such that
	\[
	\omega=\Omega
	\qquad \sigma\text{-a.e.},
	\qquad
	I(\omega)\leq\rho_{\Cspace}(\Omega)+\varepsilon.
	\]
	If $I(\omega)=0$, then \eqref{eq:C-L1} implies that $\Omega=0$ $\sigma$-almost everywhere, and the empty family of geodesic balls suffices. Henceforth assume that
	\[
	0<I(\omega)<\infty.
	\]
	Again by \eqref{eq:C-L1}, $\omega\in L^1(\Sph)$, and hence $\abs{\omega}<\infty$ $\sigma$-almost everywhere.
	
	For each $k\in\mathbb Z$, let
	\[
	E_k
	:=
	\bigl\{
	\theta\in\Sph:
	2^k<\abs{\omega(\theta)}\leq2^{k+1}
	\bigr\},
	\]
	and set
	\[
	\varepsilon_k
	:=
	I(\omega)\,2^{-|k|-2}2^{-k/2}.
	\]
	By the definition of Hausdorff content, for each $k\in\mathbb Z$ there exist an at most countable index set $J_k$ and geodesic balls
	\[
	B_{k,j}
	=
	B(\theta_{k,j},r_{k,j}),
	\qquad j\in J_k,
	\]
	such that
	\[
	E_k
	\subset
	\bigcup_{j\in J_k}B_{k,j}
	\]
	and
	\begin{equation}\label{eq:Ek-ball-cover}
		\sum_{j\in J_k}r_{k,j}^{(n-1)/2}
		\leq
		\Hc(E_k)+\varepsilon_k
		\leq
		\Hc(F_k(\omega))+\varepsilon_k.
	\end{equation}
	Define
	\[
	c_{k,j}
	:=
	2^{k+1}\sig(B_{k,j}),
	\qquad k\in\mathbb Z,\quad j\in J_k.
	\]
	Since
	\[
	u_{B_{k,j}}
	=
	\frac{\1_{B_{k,j}}}{\sig(B_{k,j})},
	\]
	we have
	\begin{equation}\label{eq:ckj-atom}
		c_{k,j}u_{B_{k,j}}
		=
		2^{k+1}\1_{B_{k,j}}.
	\end{equation}
	
	The level sets $\{E_k\}_{k\in\mathbb Z}$ are pairwise disjoint, and
	\[
	\bigcup_{k\in\mathbb Z}E_k
	=
	\bigl\{
	\theta\in\Sph:
	0<\abs{\omega(\theta)}<\infty
	\bigr\}.
	\]
	Therefore, for $\sigma$-almost every $\theta\in\Sph$,
	\begin{align}
		\abs{\omega(\theta)}
		&=
		\sum_{k\in\mathbb Z}
		\abs{\omega(\theta)}\1_{E_k}(\theta)
		\notag\\
		&\leq
		\sum_{k\in\mathbb Z}
		2^{k+1}\1_{E_k}(\theta)
		\notag\\
		&\leq
		\sum_{k\in\mathbb Z}
		\sum_{j\in J_k}
		2^{k+1}\1_{B_{k,j}}(\theta)
		\notag\\
		&=
		\sum_{k\in\mathbb Z}
		\sum_{j\in J_k}
		c_{k,j}u_{B_{k,j}}(\theta),
		\label{eq:double-majorant}
	\end{align}
	where the third line uses
	\[
	\1_{E_k}
	\leq
	\sum_{j\in J_k}\1_{B_{k,j}},
	\]
	and the last line uses \eqref{eq:ckj-atom}. Since $\omega=\Omega$ $\sigma$-almost everywhere, \eqref{eq:double-majorant} gives
	\begin{equation}\label{eq:double-majorant-Omega}
		\abs{\Omega}
		\leq
		\sum_{k\in\mathbb Z}
		\sum_{j\in J_k}
		c_{k,j}u_{B_{k,j}}
		\qquad \sigma\text{-a.e.}
	\end{equation}
	
	We next estimate the coefficient sum. By the Ahlfors regularity of the unit sphere,
	\[
	\sig(B_{k,j})^{1/2}
	\lesssim_n
	r_{k,j}^{(n-1)/2}.
	\]
	Together with \eqref{eq:Ek-ball-cover}, this gives
	\begin{align*}
		\sum_{k\in\mathbb Z}\sum_{j\in J_k}\sqrt{c_{k,j}}
		&=
		\sum_{k\in\mathbb Z}
		2^{(k+1)/2}
		\sum_{j\in J_k}\sig(B_{k,j})^{1/2}\\
		&\lesssim_n
		\sum_{k\in\mathbb Z}
		2^{k/2}
		\sum_{j\in J_k}r_{k,j}^{(n-1)/2}\\
		&\leq
		C_n\sum_{k\in\mathbb Z}
		2^{k/2}\Hc(F_k(\omega))
		+
		C_n\sum_{k\in\mathbb Z}
		2^{k/2}\varepsilon_k.
	\end{align*}
	By \eqref{eq:dyadic-layer},
	\[
	\sum_{k\in\mathbb Z}
	2^{k/2}\Hc(F_k(\omega))
	\lesssim_n
	I(\omega).
	\]
	Moreover, by the choice of $\varepsilon_k$,
	\[
	\begin{aligned}
		\sum_{k\in\mathbb Z}2^{k/2}\varepsilon_k
		&=
		I(\omega)
		\sum_{k\in\mathbb Z}2^{-|k|-2}\\
		&\lesssim I(\omega).
	\end{aligned}
	\]
	Consequently,
	\begin{equation}\label{eq:double-coefficient-sum}
		\sum_{k\in\mathbb Z}\sum_{j\in J_k}\sqrt{c_{k,j}}
		\lesssim_n
		I(\omega)
		\leq
		\rho_{\Cspace}(\Omega)+\varepsilon.
	\end{equation}
	
	Finally, the index set
	\[
	\mathcal J
	:=
	\bigcup_{k\in\mathbb Z}
	\bigl(\{k\}\times J_k\bigr)
	\]
	is at most countable. Enumerate it as
	\[
	\mathcal J
	=
	\{(k(m),j(m)):m\geq1\},
	\]
	and define
	\[
	B_m:=B_{k(m),j(m)},
	\qquad
	c_m:=c_{k(m),j(m)}.
	\]
	Since both series have nonnegative terms, they may be reindexed without altering their sums. Hence,
	\[
	\abs{\Omega}
	\leq
	\sum_{m=1}^{\infty}c_mu_{B_m}
	\qquad \sigma\text{-a.e.},
	\]
	and
	\[
	\sum_{m=1}^{\infty}\sqrt{c_m}
	\lesssim_n
	\rho_{\Cspace}(\Omega)+\varepsilon.
	\]
	This proves \eqref{eq:majorant-decomp} and \eqref{eq:majorant-coefficient-sum}.
\end{proof}

\begin{proposition}\label{prop:ball-majorant-converse}
Suppose that
\[
 \abs\Omega\le\sum_jc_ju_{B_j}
 \quad \sigma\text{-a.e.},
 \qquad c_j\ge0,
 \qquad\sum_j\sqrt{c_j}<\infty,
\]
Then $\Omega\in\Cspace(\Sph)$, and
\[
 \rho_{\Cspace}(\Omega)
 \lesssim_n\sum_j\sqrt{c_j}.
\]
\end{proposition}

\begin{proof}
Choose a measurable function $\omega=\Omega$ $\sigma$-almost everywhere and for which the above inequality holds. Since
\[
 \abs\omega^{1/2}
 \le\left(\sum_jc_ju_{B_j}\right)^{1/2}
 \le\sum_j\sqrt{c_j}\,u_{B_j}^{1/2},
\]
Lemma \ref{lem:choquet-subadd} gives
\[
 I(\omega)
 \lesssim_n
 \sum_j\sqrt{c_j}
 \int_{\Sph}u_{B_j}^{1/2}\dd\Hc.
\]
Moreover,
\[
 \int_{\Sph}u_{B_j}^{1/2}\dd\Hc
 =\frac{\Hc(B_j)}{\sig(B_j)^{1/2}}
 \asymp_n1.
\]
Hence
\[
 \rho_{\Cspace}(\Omega)\le I(\omega)
 \lesssim_n\sum_j\sqrt{c_j}.
\]
\end{proof}

\begin{proposition}\label{prop:exact-atoms}
One has
\begin{equation}\label{eq:exact-atomic-norm}
 \norm{\Omega}_{\Aspace}
 =
 \inf
 \left\{
 \left(\sum_j\sqrt{c_j}\right)^2:
 \Omega=\sum_jc_ja_j\ \text{in }L^1(\Sph),\
 c_j\ge0,\
 a_j\text{ is a geodesic-ball atom}
 \right\}.
\end{equation}
\end{proposition}

\begin{proof}
We first show that the infimum in \eqref{eq:atomic-majorant-norm} is no larger than the infimum in \eqref{eq:exact-atomic-norm}. It suffices to consider representations for which the coefficient sum on the right-hand side is finite, namely,
\[
 \Omega=\sum_jc_ja_j\quad\text{in }L^1(\Sph),
 \qquad
 \sum_j\sqrt{c_j}<\infty.
\]
In this case,
\(
 \sum_jc_j
 \le
 \left(\sum_j\sqrt{c_j}\right)^2
 <\infty.
\)
Since every geodesic-ball atom satisfies $\norm{a_j}_{L^1}\le1$,
\(
 \sum_jc_j\norm{a_j}_{L^1}<\infty,
\)
and hence the series converges absolutely in $L^1$. A subsequence of its partial sums therefore converges to $\Omega$ almost everywhere, and
\[
 \abs\Omega
 \le\sum_jc_j\abs{a_j}
 \le\sum_jc_ju_{B_j}
 \quad\text{a.e.}
\]
Thus the infimum in \eqref{eq:atomic-majorant-norm} is no larger than that in \eqref{eq:exact-atomic-norm}.

Conversely, suppose that
\(
 \abs\Omega\le G:=\sum_jc_ju_{B_j}
 \quad\text{a.e.}
\)
Since $\sum_jc_j\le(\sum_j\sqrt{c_j})^2<\infty$,
$G\in L^1$. Define
\[
 a_j(\theta)
 :=
 \begin{cases}
 \dfrac{\Omega(\theta)}{G(\theta)}u_{B_j}(\theta),
 &G(\theta)>0,\\[1ex]
 0,&G(\theta)=0.
 \end{cases}
\]
Then $\abs{a_j}\le u_{B_j}$, and
\(
 \sum_jc_ja_j=\Omega
 \quad\text{a.e.}
\)
Moreover,
\(
 \sum_jc_j\norm{a_j}_{L^1}
 \le\sum_jc_j<\infty,
\)
so the series converges absolutely in $L^1$. This shows that the two infima are coincide.
\end{proof}

\begin{proof}[Proof of Theorem \ref{thm:atomic-main}]
Proposition \ref{prop:ball-majorant} gives
\[
 \norm{\Omega}_{\Aspace}^{1/2}
 \lesssim_n\rho_{\Cspace}(\Omega).
\]
Proposition \ref{prop:ball-majorant-converse} gives the reverse estimate. Therefore,
\[
 \norm{\Omega}_{\Aspace}\asymp_n\norm{\Omega}_{\Cspace}.
\]

By definition, $\norm{\cdot}_{\Aspace}^{1/2}$ satisfies the triangle inequality, and if $\abs f\le\abs g$ almost everywhere, then $\norm f_{\Aspace}\le\norm g_{\Aspace}$. Moreover,
\begin{equation}\label{eq:A-L1}
 \norm{\Omega}_{L^1}
 \le\norm{\Omega}_{\Aspace},
\end{equation}
since 
\[
 \norm{\Omega}_{L^1}
 \le\sum_jc_j
 \le\left(\sum_j\sqrt{c_j}\right)^2.
\]

We now turn to the proof of completeness. Let $\{\Omega_m\}$ be a Cauchy sequence with respect to $\norm{\cdot}_{\Aspace}^{1/2}$. Choose a subsequence $\{\Omega_{m_k}\}$ such that
\(
 \norm{\Omega_{m_{k+1}}-\Omega_{m_k}}_{\Aspace}^{1/2}
 \le2^{-k}.
\)
By Proposition \ref{prop:exact-atoms}, we may write $h_k:=\Omega_{m_{k+1}}-\Omega_{m_k}$ as
\(
 h_k=\sum_jc_{k,j}a_{k,j}
\)
with
\(
 \sum_j\sqrt{c_{k,j}}\le2^{1-k}.
\)
Thus
\[
 \norm{h_k}_{L^1}
 \le\sum_jc_{k,j}
 \le\left(\sum_j\sqrt{c_{k,j}}\right)^2
 \le4^{1-k}.
\]
Hence $\sum_kh_k$ converges in $L^1$. Let
\(
 \Omega:=\Omega_{m_1}+\sum_{k=1}^\infty h_k.
\)
Combining all atomic representations in the tail yields
\[
 \norm{\Omega-\Omega_{m_N}}_{\Aspace}^{1/2}
 \le\sum_{k\ge N}\sum_j\sqrt{c_{k,j}}
 \lesssim2^{-N}.
\]
Thus the subsequence converges to $\Omega$ in $\Aspace$, and so does the original Cauchy sequence. This establishes that $\Aspace$ is complete. By equivalence of quasi-norms, $\Cspace$ is also complete. The triangle inequality above also shows that $\norm{\cdot}_{\Aspace}$ is a $1/2$-norm. Consequently, $\Cspace$ is a complete quasi-Banach space with respect to $\norm{\cdot}_{\Cspace}$ and admits an equivalent $1/2$-norm.

Finally, for every exact atomic representation $\Omega=\sum_jc_ja_j$, the finite partial sums satisfy
\[
 \norm{\Omega-\sum_{j=1}^Nc_ja_j}_{\Aspace}^{1/2}
 \le\sum_{j>N}\sqrt{c_j}\longrightarrow0.
\]
Hence finite atomic sums are dense.
\end{proof}

\subsection{Comparison with Classical Kernel Spaces}\label{sec:comparisons}

In this subsection we prove Theorem \ref{thm:comparison-main}. We first locate $\Cspace$ within the Lorentz scale and show that its embedding into $L^{1,1/2}$ is strict. We then construct a kernel that belongs to $\Cspace$ but belongs to neither $L\log^+\! L$ nor any $L^p$ for $p>1$.

\subsubsection{Basic embeddings}

Here $L^{1,1/2}(\Sph)$ is equipped with the distribution-function quasi-norm from Lemma \ref{lem:lorentz-L1}.

\begin{proposition}\label{prop:lorentz-embedding}
There are continuous embeddings
\[
 L^\infty(\Sph)\subset\Cspace(\Sph)
 \subset L^{1,1/2}(\Sph)\subset L^1(\Sph).
\]
More precisely,
\[
 \norm{\Omega}_{L^{1,1/2}}
 \lesssim_n\norm{\Omega}_{\Cspace},
 \qquad
 \norm{\Omega}_{L^1}
 \lesssim_n\norm{\Omega}_{\Cspace}.
\]
\end{proposition}

\begin{proof}
If $\Omega\in L^\infty$, then
\[
 \rho_{\Cspace}(\Omega)
 \le \Hc(\Sph)\norm{\Omega}_{L^\infty}^{1/2}<\infty.
\]
Hence $L^\infty\subset\Cspace$.

Conversely, if $\omega=\Omega$ almost everywhere, then Proposition \ref{prop:content-estimates} and the level-set formula \eqref{eq:continuous-layer} give
\begin{align*}
 I(\omega)
 &\gtrsim_n
 \frac12\int_0^\infty
 t^{-1/2}\sig(\{\abs\omega>t\})^{1/2}\dd t\\
 &=
 \frac12\int_0^\infty
 t^{-1/2}\sig(\{\abs\Omega>t\})^{1/2}\dd t.
\end{align*}
Taking the infimum over all measurable functions $\omega=\Omega$ $\sigma$-almost everywhere and then squaring yields $\Cspace\subset L^{1,1/2}$. Finally, Lemma \ref{lem:lorentz-L1} gives $L^{1,1/2}\subset L^1$.
\end{proof}

\subsubsection{Strictness of the Lorentz embedding}

The preceding proposition gives the continuous embedding
$\Cspace(\Sph)\subset L^{1,1/2}(\Sph)$. We next show that this inclusion is strict. The construction uses sets of small surface measure whose $(n-1)/2$-dimensional Hausdorff content remains uniformly positive.

\begin{proposition}\label{prop:C-strict-Lorentz}
There exists a nonnegative measurable function
\[
 w\in L^{1,1/2}(\Sph)\setminus\Cspace(\Sph).
\]
Consequently,
\[
 \Cspace(\Sph)\subsetneq L^{1,1/2}(\Sph).
\]
\end{proposition}

\begin{proof}
We first construct a family of sets with arbitrarily small surface measure but uniformly positive $(n-1)/2$-dimensional Hausdorff content. Fix $0<r<r_0$, where $r_0>0$ is sufficiently small, and choose an $r$-separated $2r$-net
\[
 \{\eta_1,\ldots,\eta_M\}\subset\Sph.
\]
that is,
\[
d_{\Sph}(\eta_i,\eta_j)\ge r \quad (i\ne j),
\qquad
\Sph\subset \bigcup_{j=1}^{M} B(\eta_j,2r).
\]
By Lemma \ref{lem:ahlfors} and the standard packing estimate,
\[
 M\asymp_n r^{-(n-1)}.
\]
Set
\[
 s:=r^2,
 \qquad
 B_j:=B(\eta_j,s),
 \qquad
 E_r:=\bigcup_{j=1}^M B_j.
\]
After decreasing $r_0$ if necessary, the balls $B_j$ are pairwise disjoint. Hence Lemma \ref{lem:ahlfors} gives
\begin{equation}\label{eq:Er-sigma}
 \sig(E_r)
 \asymp_n M s^{n-1}
 \asymp_n r^{n-1}.
\end{equation}
On the other hand, using the balls $B_j$ themselves as a covering,
\begin{equation}\label{eq:Er-H-upper}
 \Hc(E_r)
 \le \sum_{j=1}^M s^{(n-1)/2}
 \asymp_n 1.
\end{equation}

For the reverse estimate, define a probability measure supported on $E_r$ by
\[
 \nu_r
 :=
 \frac1M\sum_{j=1}^M
 \frac{\sig|_{B_j}}{\sig(B_j)}.
\]
We claim that
\begin{equation}\label{eq:nu-r-frostman}
 \nu_r(B(\xi,\rho))\lesssim_n \rho^{(n-1)/2},
 \qquad \xi\in\Sph,\quad 0<\rho\le1.
\end{equation}
Indeed, if $0<\rho\le s$, then $B(\xi,\rho)$ meets only $O_n(1)$ of the balls $B_j$, and Lemma \ref{lem:ahlfors} yields
\[
 \nu_r(B(\xi,\rho))
 \lesssim_n
 \frac1M\frac{\rho^{n-1}}{s^{n-1}}
 \asymp_n
 \frac{\rho^{n-1}}{r^{n-1}}
 \le \rho^{(n-1)/2},
\]
since $\rho\le r^2$. If $s<\rho<r$, the $r$-separation again implies that $B(\xi,\rho)$ meets only $O_n(1)$ of the balls $B_j$, and therefore
\[
 \nu_r(B(\xi,\rho))
 \lesssim_n M^{-1}
 \asymp_n r^{n-1}
 \le \rho^{(n-1)/2},
\]
because $\rho>r^2$. Finally, if $r\le\rho\le1$, the packing estimate gives
\[
 \#\{j:B_j\cap B(\xi,\rho)\neq\varnothing\}
 \lesssim_n \left(\frac{\rho}{r}\right)^{n-1},
\]
and hence
\[
 \nu_r(B(\xi,\rho))
 \lesssim_n
 \frac1M\left(\frac{\rho}{r}\right)^{n-1}
 \asymp_n \rho^{n-1}
 \le \rho^{(n-1)/2}.
\]
This proves \eqref{eq:nu-r-frostman}.

Let $E_r\subset\bigcup_\ell B(\xi_\ell,\rho_\ell)$ be any admissible covering in the definition of $\Hc$. Since $\nu_r(E_r)=1$, \eqref{eq:nu-r-frostman} gives
\[
 1
 \le \sum_\ell \nu_r(B(\xi_\ell,\rho_\ell))
 \lesssim_n \sum_\ell \rho_\ell^{(n-1)/2}.
\]
Taking the infimum over all such coverings and combining with \eqref{eq:Er-H-upper}, we obtain
\begin{equation}\label{eq:Er-H-two-sided}
 \Hc(E_r)\asymp_n1.
\end{equation}
Moreover, if $N\subset\Sph$ satisfies $\sig(N)=0$, then $\nu_r(N)=0$, since $\nu_r\ll\sig$. Thus $\nu_r(E_r\setminus N)=1$, and the same covering argument gives
\begin{equation}\label{eq:Er-null-robust}
 \Hc(E_r\setminus N)\gtrsim_n1,
\end{equation}
with a constant independent of $r$ and $N$.

We now combine these sets at different scales. Choose $r_k\downarrow0$ so that
\[
 r_k^{n-1}=2^{-4k},
\]
and write $E_k:=E_{r_k}$. By \eqref{eq:Er-sigma},
\begin{equation}\label{eq:Ek-sigma}
 \sig(E_k)\lesssim_n2^{-4k},
\end{equation}
while \eqref{eq:Er-null-robust} implies that, for every $\sig$-null set $N$,
\begin{equation}\label{eq:Ek-null-robust}
 \Hc(E_k\setminus N)\gtrsim_n1,
\end{equation}
with a constant independent of $k$.

Define
\[
 w(\theta):=\sup_{k\ge1}2^k\1_{E_k}(\theta).
\]
Since $\sum_k\sig(E_k)<\infty$, the Borel--Cantelli lemma shows that $w<\infty$ for $\sig$-almost every $\theta$.

We first verify that $w\in L^{1,1/2}(\Sph)$. If $2^m\le t<2^{m+1}$, then
\[
 \{w>t\}\subset\bigcup_{k\ge m+1}E_k,
\]
and therefore \eqref{eq:Ek-sigma} gives
\[
 \sig(\{w>t\})
 \lesssim_n\sum_{k\ge m+1}2^{-4k}
 \lesssim_n2^{-4m}.
\]
Using the distribution-function quasi-norm from Lemma \ref{lem:lorentz-L1},
\begin{align*}
 \norm{w}_{L^{1,1/2}}^{1/2}
 &\lesssim_n
 1+
 \sum_{m=1}^\infty
 \int_{2^m}^{2^{m+1}}
 t^{-1/2}2^{-2m}\dd t\\
 &\lesssim_n
 1+\sum_{m=1}^\infty2^{-3m/2}
 <\infty.
\end{align*}
Thus $w\in L^{1,1/2}(\Sph)$.

It remains to show that $w\notin\Cspace(\Sph)$. Let $\omega=w$ $\sig$-almost everywhere and set
\[
 N:=\{\theta\in\Sph:\omega(\theta)\neq w(\theta)\}.
\]
Then $\sig(N)=0$. For every $k\ge1$,
\[
 E_k\setminus N\subset\{\abs\omega\ge2^k\}.
\]
Hence \eqref{eq:Ek-null-robust} gives
\[
 \Hc(\{\abs\omega>2^{k-1}\})\gtrsim_n1.
\]
Using \eqref{eq:continuous-layer},
\[
 I(\omega)
 \ge
 \frac12\int_0^{2^{k-1}}
 t^{-1/2}\Hc(\{\abs\omega>t\})\dd t
 \gtrsim_n2^{k/2}.
\]
Since $k$ is arbitrary, $I(\omega)=\infty$. This holds for every measurable representative $\omega=w$ $\sig$-almost everywhere. Therefore
\[
 \rho_{\Cspace}(w)=\infty,
\]
and hence $w\notin\Cspace(\Sph)$.
\end{proof}

\subsubsection{An example beyond the $L\log^+ L$ and $L^p$ conditions}

We now construct a nonnegative kernel that belongs to $\Cspace$, but belongs to neither  $L\log^+ L$ nor  any $L^p$ for $p>1$.

\begin{proposition}\label{prop:C-not-LlogL}
There exists a nonnegative kernel
\[
 \Omega_*\in\Cspace(\Sph)
 \setminus\bigcup_{p>1}L^p(\Sph),
 \qquad
 \Omega_*\notin L\log^+ L(\Sph).
\]
\end{proposition}

\begin{proof}
Let
\[
 c_k:=\frac{4^{-k}}{k^2},
 \qquad
 r_k:=\varepsilon\exp\left(-\frac{k4^k}{n-1}\right),
 \qquad k\ge1.
\]
Set
\[
\gamma(s):=(\cos s,\sin s,0,\ldots,0),
\qquad 0\le s\le \frac12.
\]
Then, for every $s,t\in[0,1/2]$,
\[
d_{\Sph}(\gamma(s),\gamma(t))=|s-t|.
\]
Let $\theta_k:=\gamma(2^{-k})$. Choose $\varepsilon>0$ sufficiently small that $r_k\le2^{-k-3}$ for every $k$. Since $r_k$ is decreasing, if $k<\ell$, then
\[
 d_{\Sph}(\theta_k,\theta_\ell)
 =2^{-k}-2^{-\ell}
 \ge2^{-k-1}
 >r_k+r_\ell.
\]
Thus the geodesic balls $B_k:=B(\theta_k,r_k)$ are pairwise disjoint.

Define
\[
 b_k:=\frac{c_k}{\sig(B_k)},
 \qquad
 \Omega_*:=\sum_{k=1}^\infty b_k\1_{B_k}
 =\sum_{k=1}^\infty c_ku_{B_k}.
\]
Since
\[
 \sum_{k=1}^\infty\sqrt{c_k}
 =\sum_{k=1}^\infty\frac{2^{-k}}k<\infty,
\]
Proposition \ref{prop:ball-majorant-converse} gives
$\Omega_*\in\Cspace(\Sph)$.

By Lemma \ref{lem:ahlfors},
\(
 \sig(B_k)\asymp_n r_k^{n-1}
 =\varepsilon^{n-1} e^{-k4^k}.
\)
Therefore, for every fixed $a>0$ and all sufficiently large $k$,
\(
 \log\left(2+\frac{b_k}{a}\right)
 \gtrsim_a k4^k.
\)
Using the pairwise disjointness of the $B_k$,
\[
 \int_{\Sph}
 \Phi\left(\frac{\Omega_*}{a}\right)\dd\sig
 =
 \sum_k
 \frac{c_k}{a}
 \log\left(2+\frac{b_k}{a}\right)
 \gtrsim_a
 \sum_{k\gg1}\frac1k
 =\infty.
\]
This holds for every $a>0$, so the definition of the norm in \eqref{eq:LlogL-lux} implies that
\(
 \Omega_*\notin L\log^+ L(\Sph).
\)

Finally, fix $p>1$. Since the geodesic balls $B_k$ are pairwise disjoint,
\(
	\|\Omega_*\|_{L^p}^p=\sum_{k=1}^{\infty}c_k^p\sigma(B_k)^{1-p}.
\)
Moreover,
\[
c_k^p\sigma(B_k)^{1-p}
\asymp_{n,p,\varepsilon}
4^{-pk}k^{-2p}e^{(p-1)k4^k}.
\]
Since $p>1$, the exponential factor $e^{(p-1)k4^k}$ dominates $4^{pk}k^{2p}$ for large $k$, and consequently
\[
c_k^p\sigma(B_k)^{1-p}\longrightarrow+\infty.
\]
Thus the series diverges, so
\[
\Omega_*\notin L^p(\Sph).
\]
\end{proof}

\begin{proof}[Proof of Theorem \ref{thm:comparison-main}]
Proposition \ref{prop:lorentz-embedding} yields the continuous embeddings, Proposition \ref{prop:C-strict-Lorentz} shows that the inclusion $\Cspace(\Sph)\subset L^{1,1/2}(\Sph)$ is strict, and Proposition \ref{prop:C-not-LlogL} provides the required nonnegative kernel $\Omega_*$.
\end{proof}

\section{Weak-Type Estimates for Rough Maximal Operators and the Sharp Scale}\label{sec:weak-main}

\subsection{A Uniform Weak Type Estimate for a Single Geodesic-Ball Kernel}\label{subsec:single-geodesic-ball}

For a geodesic ball $B\subset\Sph$ and $R>0$, set
\[
 \Gamma(B,R):=\{\rho\theta:0<\rho<R,\ \theta\in B\}.
\]

\begin{lemma}\label{lem:cone-box}
There exist constants $r_0\in(0,1)$ and $C_n>1$, depending only on $n$, such that if
$B=B(e,r)$ and $0<r\le r_0$, then for every $R>0$, with respect to the orthogonal decomposition
\(
 \R^n=\R e\oplus e^\perp
\)
one has
\(
 \Gamma(B,R)\subset Q(B,R),
\)
where
\[
 Q(B,R):=\{se+z:0<s<C_nR,\ z\in e^\perp,\ \abs z<C_nrR\},
\]
and
\[
 \abs{Q(B,R)}\asymp_nr^{n-1}R^n\asymp_n\sig(B)R^n.
\]
\end{lemma}

\begin{proof}
Take any $y=\rho\theta\in\Gamma(B,R)$. Let
\[
 s:=\rho\langle\theta,e\rangle,
 \qquad
 z:=\rho\bigl(\theta-\langle\theta,e\rangle e\bigr).
\]
Then $z\perp e$ and $y=se+z$. Since $d_{\Sph}(\theta,e)<r\le r_0$, choosing $r_0<\pi/3$ ensures that $\langle\theta,e\rangle>0$, and hence
\[
 0<s=\rho\langle\theta,e\rangle\le\rho<R.
\]
Moreover, if $\gamma=d_{\Sph}(\theta,e)$, then
\[
 \abs z=\rho\sin\gamma\le R\sin r\le Rr.
\]
Thus $y\in Q(B,R)$, proving the inclusion.

The length of $Q(B,R)$ in the $e$ direction is comparable to $R$, while its cross-section in the $n-1$ transverse directions is an $(n-1)$-dimensional ball of radius comparable to $rR$. Hence
\[
 \abs{Q(B,R)}\asymp_nR(rR)^{n-1}=r^{n-1}R^n.
\]
The final comparison follows from \eqref{eq:ball-measure}, since $r^{n-1}\asymp_n\sig(B)$.
\end{proof}

Before establishing the single-geodesic-ball estimate, it is important to distinguish it from the Kakeya maximal function. Although the argument below also uses averages over elongated regions, the geodesic ball $B=B(e,r)$ remains fixed throughout; both its central direction $e$ and geodesic radius $r$ are fixed, and only the radial scale $R>0$ varies. By contrast, the Kakeya maximal function takes a supremum over all possible directions of the elongated regions; see \cite[Chapter~X]{Stein1993}. Thus, the single-geodesic-ball estimate below does not imply a weak-type $(1,1)$ estimate for the Kakeya maximal function.

\begin{theorem}[Single-geodesic-ball estimate]\label{thm:single-geodesic-ball}
There exists a constant $C_n$ such that for every geodesic ball $B\subset\Sph$, every
$f\in L^1(\R^n)$, and every $\lambda>0$,
\[
 \abs{\{x:\mathcal M_{u_B}f(x)>\lambda\}}
 \le\frac{C_n}{\lambda}\norm f_1.
\]
The constant is independent of the location and radius of $B$.
\end{theorem}

\begin{proof}
Let $B=B(e,r)$. Then
\begin{align}
 \mathcal M_{u_B}f(x)
 &=c_n\sup_{R>0}\frac1{R^n\sig(B)}
 \int_0^R\int_B
 \abs{f(x-\rho\theta)}\rho^{n-1}\dd\sig(\theta)\dd\rho\notag\\
 &=c_n\sup_{R>0}\frac1{R^n\sig(B)}
 \int_{\Gamma(B,R)}\abs{f(x-y)}\dd y.
 \label{eq:ball-polar}
\end{align}

First suppose that $0<r\le r_0$, where $r_0$ is as in Lemma \ref{lem:cone-box}. By that lemma,
\[
 \Gamma(B,R)\subset Q(B,R),
 \qquad
 R^n\sig(B)\asymp_n\abs{Q(B,R)}.
\]
Hence
\begin{equation}\label{eq:ball-to-box-max}
 \mathcal M_{u_B}f(x)
 \lesssim_n
 \sup_{R>0}\frac1{\abs{Q(B,R)}}
 \int_{Q(B,R)}\abs{f(x-y)}\dd y.
\end{equation}

Define the invertible linear map
\[
 L_B(se+z):=se+r^{-1}z,
 \qquad s\in\R,\ z\in e^\perp.
\]
It leaves the $e$ direction unchanged and dilates each of the $n-1$ directions in $e^\perp$ by the factor $r^{-1}$. Therefore,
\[
 J_B:=\abs{\det L_B}=r^{-(n-1)},
 \qquad
 \abs{\det L_B^{-1}}=r^{n-1}.
\]
Moreover,
\[
 L_BQ(B,R)=RQ_0,
\]
where
\[
Q_0
:=
\bigl\{
se+w:
0<s<C_n,\;
w\in e^\perp,\;
|w|<C_n
\bigr\},
\]
and $RQ_0$ denotes the dilation of $Q_0$ by the factor $R$ about the origin; namely,
\[
\begin{aligned}
	RQ_0
	&:=
	\{Ry:y\in Q_0\}\\
	&=
	\bigl\{
	se+w:
	0<s<C_nR,\;
	w\in e^\perp,\;
	|w|<C_nR
	\bigr\}.
\end{aligned}
\]

Let
\[
 F(w):=f(L_B^{-1}w),
 \qquad x'=L_Bx.
\]
In the integral in \eqref{eq:ball-to-box-max}, make the change of variables $w=L_By$. Since $\dd y=J_B^{-1}\dd w$ and $\abs{Q(B,R)}=J_B^{-1}\abs{RQ_0}$, we obtain
\begin{align*}
 &\frac1{\abs{Q(B,R)}}
 \int_{Q(B,R)}\abs{f(x-y)}\dd y\\
 &\qquad=
 \frac1{\abs{RQ_0}}
 \int_{RQ_0}
 \abs{f(L_B^{-1}x'-L_B^{-1}w)}\dd w\\
 &\qquad=
 \frac1{\abs{RQ_0}}
 \int_{RQ_0}\abs{F(x'-w)}\dd w.
\end{align*}
Therefore,
\[
 \mathcal M_{u_B}f(L_B^{-1}x')
 \lesssim_n\mathcal N_{Q_0}F(x').
\]
Consequently by Lemma \ref{lem:fixed-shape},
\begin{align*}
 \abs{\{x:\mathcal M_{u_B}f(x)>\lambda\}}
 &\le J_B^{-1}
 \abs{\{x':C_n\mathcal N_{Q_0}F(x')>\lambda\}}\\
 &\lesssim_nJ_B^{-1}\frac{\norm F_1}{\lambda}.
\end{align*}
On the other hand, the change of variables $w=L_Bx$ gives
\[
 \norm F_1
 =\int_{\R^n}\abs{f(L_B^{-1}w)}\dd w
 =J_B\norm f_1.
\]
Thus, the two Jacobian factors cancel out exactly, yielding
\[
 \abs{\{x:\mathcal M_{u_B}f(x)>\lambda\}}
 \lesssim_n\frac{\norm f_1}{\lambda}.
\]

Finally, suppose that $r>r_0$. By Lemma \ref{lem:ahlfors},
$\sig(B)\ge\sig(B(e,r_0))\gtrsim_n1$, and
$\Gamma(B,R)\subset B_{\R^n}(0,R)$. By \eqref{eq:ball-polar},
\[
 \mathcal M_{u_B}f(x)
 \lesssim_n
 \sup_{R>0}\frac1{R^n}
 \int_{\abs y<R}\abs{f(x-y)}\dd y
 \lesssim_n\mathcal Mf(x).
\]
Applying Lemma \ref{lem:HL} once again, we obtain the desired result.
\end{proof}

\subsection{Weak-Type Estimates for General Hausdorff--Choquet Kernels}\label{subsec:C-weak}

\begin{proof}[Proof of Theorem \ref{thm:main}]
Fix $\varepsilon>0$. By Proposition \ref{prop:ball-majorant}, there exist geodesic balls $B_j\subset\Sph$ and coefficients $c_j\ge0$ such that
\[
 \abs{\Omega}\le\sum_{j=1}^\infty c_ju_{B_j}
 \quad\text{a.e.},
 \qquad
 \sum_{j=1}^\infty\sqrt{c_j}
 \lesssim_n\rho_{\Cspace}(\Omega)+\varepsilon.
\]
Then we get
\[
 \mathcal M_{\Omega}f
 \le\sum_{j=1}^\infty c_j\mathcal M_{u_{B_j}}f.
\]

For any $N\in\mathbb N$, set
\[
 F_N:=\sum_{j=1}^Nc_j\mathcal M_{u_{B_j}}f,
 \qquad
 S_N:=\sum_{j=1}^N\sqrt{c_j}.
\]
If $S_N=0$, the conclusion is immediate. Otherwise, discard all terms with $c_j=0$ and sum only over indices for which $c_j>0$. Let
\[
 \lambda_j:=\lambda\frac{\sqrt{c_j}}{S_N}.
\]
It then follows that the sum of these $\lambda_j$ is $\lambda$, and
\[
 \{F_N>\lambda\}
 \subset
 \bigcup_{\substack{1\le j\le N\\ c_j>0}}
 \left\{
 \mathcal M_{u_{B_j}}f>\frac{\lambda_j}{c_j}
 \right\}.
\]
By Theorem \ref{thm:single-geodesic-ball},
\begin{align*}
 \abs{\{F_N>\lambda\}}
 &\lesssim_n
 \norm f_1\sum_{\substack{1\le j\le N\\ c_j>0}}\frac{c_j}{\lambda_j}\\
 &=\frac{S_N^2}{\lambda}\norm f_1\\
 &\lesssim_n
 \frac{(\rho_{\Cspace}(\Omega)+\varepsilon)^2}{\lambda}
 \norm f_1.
\end{align*}
Letting $N\to\infty$ and using monotone convergence yields the same estimate; since $\mathcal M_{\Omega}f\le\lim_NF_N$, we obtain
\[
 \abs{\{x:\mathcal M_{\Omega}f(x)>\lambda\}}
 \lesssim_n
 \frac{(\rho_{\Cspace}(\Omega)+\varepsilon)^2}{\lambda}
 \norm f_1.
\]
Finally, letting $\varepsilon\downarrow0$ gives
\[
 \abs{\{x:\mathcal M_{\Omega}f(x)>\lambda\}}
 \lesssim_n
 \frac{\norm{\Omega}_{\Cspace}}{\lambda}\norm f_1.
\]
\end{proof}

\subsection{The Sum Space with $L\log^+\! L$}\label{subsec:sum-space}

By Lemma \ref{lem:classical-llogl} and Theorem \ref{thm:main}, both the $L\log^+\! L$ and $\Cspace$ components enjoy homogeneous weak-type $(1,1)$ bounds. We now prove Theorem \ref{thm:bridge} directly. The sum space $\Xspace$ and the control quantity $\mathcal N_{\Xspace}$ are as defined in \eqref{eq:X-control-main}.

\begin{proof}[Proof of Theorem \ref{thm:bridge}]
Take any $\Omega\in\Xspace(\Sph)$ and any admissible decomposition
\[
 \Omega=\Omega_1+\Omega_2,
 \qquad
 \Omega_1\in L\log^+\! L(\Sph),
 \qquad
 \Omega_2\in\Cspace(\Sph).
\]
Since
\[
 \abs\Omega\le\abs{\Omega_1}+\abs{\Omega_2},
\]
the elementary properties of the maximal operator give
\[
 \mathcal M_{\Omega}f\le\mathcal M_{\Omega_1}f+\mathcal M_{\Omega_2}f.
\]
Therefore, for every $\lambda>0$,
\[
 \{x:\mathcal M_{\Omega}f(x)>\lambda\}
 \subset
 \left\{x:\mathcal M_{\Omega_1}f(x)>\frac\lambda2\right\}
 \cup
 \left\{x:\mathcal M_{\Omega_2}f(x)>\frac\lambda2\right\}.
\]
By Lemma \ref{lem:classical-llogl} and Theorem \ref{thm:main},
\begin{align*}
 \abs{\{x:\mathcal M_{\Omega}f(x)>\lambda\}}
 &\le
 \abs{\left\{x:\mathcal M_{\Omega_1}f(x)>\frac\lambda2\right\}}
 +
 \abs{\left\{x:\mathcal M_{\Omega_2}f(x)>\frac\lambda2\right\}}\\
 &\lesssim_n
 \frac{
 \norm{\Omega_1}_{L\log^+\! L}+\norm{\Omega_2}_{\Cspace}
 }{\lambda}\norm f_1.
\end{align*}
Taking the infimum over all admissible decompositions $\Omega=\Omega_1+\Omega_2$, we obtain
\[
 \abs{\{x:\mathcal M_{\Omega}f(x)>\lambda\}}
 \lesssim_n
 \frac{\mathcal N_{\Xspace}(\Omega)}{\lambda}\norm f_1,
\]
which is \eqref{eq:X-weak-main}.

Finally, Proposition \ref{prop:C-not-LlogL} constructs
\[
 \Omega_*\in\Cspace(\Sph)\setminus L\log^+\! L(\Sph).
\]
Since $\Cspace(\Sph)\subset\Xspace(\Sph)$, we have $\Omega_*\in\Xspace(\Sph)$, while $\Omega_*\notin L\log^+\! L(\Sph)$. Hence
\[
 L\log^+\! L(\Sph)\subsetneq\Xspace(\Sph).
\]
\end{proof}

\subsection{The Sharp Dimension in the Hausdorff--Choquet Scale}\label{subsec:critical-alpha}

In this subsection we return to the general parameter $0<\alpha\le n-1$ introduced in  Definition \ref{def:hausdorff-content} and use the Hausdorff--Choquet scale $\Cspace_\alpha$ from Definition \ref{def:C-alpha}.

\begin{proposition}[Basic estimates in general dimension]\label{prop:general-alpha-ball}
Let $0<\alpha\le n-1$. Then
\begin{equation}\label{eq:general-content-measure}
 \mathcal H^\alpha_1(E)
 \gtrsim_{n,\alpha}\sig(E)^{\alpha/(n-1)}
\end{equation}
holds for every measurable set $E\subset\Sph$, and
\begin{equation}\label{eq:general-alpha-ball}
 \mathcal H^\alpha_1(B(\theta,r))
 \asymp_{n,\alpha}r^\alpha.
\end{equation}
Consequently, for the normalized geodesic-ball kernel
\[
 u_{B(\theta,r)}
 =\frac{\1_{B(\theta,r)}}{\sig(B(\theta,r))},
\]
one has
\begin{equation}\label{eq:C-alpha-normalized-ball}
 \norm{u_{B(\theta,r)}}_{\Cspace_\alpha}
 \asymp_{n,\alpha}r^{2\alpha-(n-1)}.
\end{equation}
\end{proposition}

\begin{proof}
If $E\subset\bigcup_jB(\theta_j,r_j)$, then
\[
 \sig(E)\lesssim_n\sum_jr_j^{n-1}
 =\sum_j(r_j^\alpha)^{(n-1)/\alpha}
 \le\left(\sum_jr_j^\alpha\right)^{(n-1)/\alpha}.
\]
Taking the infimum gives \eqref{eq:general-content-measure}. Applying this with $E=B(\theta,r)$ yields $\mathcal H^\alpha_1(B(\theta,r))\gtrsim r^\alpha$, while the geodesic ball itself provides the reverse upper bound.

Taking the pointwise-defined function $u_{B(\theta,r)}$ in the definition, we obtain
\[
 \rho_{\Cspace_\alpha}(u_{B(\theta,r)})
 \lesssim
 \sig(B(\theta,r))^{-1/2}r^\alpha
 \asymp r^{\alpha-(n-1)/2}.
\]
Conversely, if $\omega=u_{B(\theta,r)}$ almost everywhere, then the set
\[
 E_\omega
 :=
 \left\{\abs\omega>
 \frac1{2\sig(B(\theta,r))}\right\}
\]
has the same $\sigma$-measure as $B(\theta,r)$. By \eqref{eq:general-content-measure},
\[
 \mathcal H^\alpha_1(E_\omega)\gtrsim r^\alpha.
\]
Hence
\[
 \int_{\Sph}\abs\omega^{1/2}\dd\mathcal H^\alpha_1
 \ge
 \frac{\mathcal H^\alpha_1(E_\omega)}
 {\sqrt{2\sig(B(\theta,r))}}
 \gtrsim r^{\alpha-(n-1)/2}.
\]
Taking the infimum over all measurable functions $\omega=u_{B(\theta,r)}$ $\sigma$-almost everywhere and then squaring gives \eqref{eq:C-alpha-normalized-ball}.
\end{proof}

\begin{theorem}\label{thm4.4}
	Let $0<\alpha\le n-1$.
	\begin{enumerate}[label=\textup{(\roman*)}]
		\item If $0<\alpha\le (n-1)/2$, then there exists a constant $C_n$, depending only on $n$, such that
		\[
		\abs{\{x:\mathcal M_{\Omega}f(x)>\lambda\}}
		\le
		\frac{C_n}{\lambda}
		\norm{\Omega}_{\Cspace_\alpha}\norm f_1.
		\]
		\item If $(n-1)/2<\alpha\le n-1$, then there is no finite constant $C_{n,\alpha}$ for which the estimate above holds simultaneously for all nonnegative measurable kernels, all $f\in L^1(\R^n)$, and all $\lambda>0$.
	\end{enumerate}
\end{theorem}

\begin{proof}
First suppose that $0<\alpha\le (n-1)/2$. Since all covering radii are at most $1$,
\[
 r^{(n-1)/2}\le r^\alpha,
\]
and hence
\[
 \mathcal H^{(n-1)/2}_1(E)\le\mathcal H^\alpha_1(E)
\]
for every $E$. Applying the level-set definition to every measurable function $\omega=\Omega$ $\sigma$-almost everywhere, and then taking the infimum over all such $\omega$ yields
\[
 \rho_{\Cspace}(\Omega)
 \le\rho_{\Cspace_\alpha}(\Omega),
 \qquad
 \norm{\Omega}_{\Cspace}
 \le\norm{\Omega}_{\Cspace_\alpha}.
\]
The sufficiency thus follows from Theorem \ref{thm:main}.

Now suppose that $(n-1)/2<\alpha\le n-1$. Suppose, for contradiction, that there exists a uniform constant $C_{n,\alpha}$. Fix $\theta_0\in\Sph$ and let
\[
 B_r:=B(\theta_0,r),
 \qquad
 u_r:=\frac{\1_{B_r}}{\sig(B_r)}.
\]
By \eqref{eq:C-alpha-normalized-ball},
\[
 \norm{u_r}_{\Cspace_\alpha}
 \asymp_{n,\alpha}r^{2\alpha-(n-1)}\longrightarrow0
 \qquad(r\downarrow0).
\]

Take $f=\1_{B_{\R^n}(0,2)}$. If $x\in B_{\R^n}(0,1)$, choose $R=1$ in \eqref{eq:rough-maximal}. Then $f(x-y)=1$ for every $\abs y<1$. Hence there exists $c_n>0$ such that
\[
 \mathcal M_{u_r}f(x)
 \ge c_n\int_{\Sph}u_r\dd\sig
 =c_n.
\]
Thus
\[
 B_{\R^n}(0,1)\subset\{x:\mathcal M_{u_r}f(x)>c_n/2\}.
\]
If the uniform estimate held, then
\[
 \abs{B_{\R^n}(0,1)}
 \le\frac{2C_{n,\alpha}}{c_n}
 \norm{u_r}_{\Cspace_\alpha}\norm f_1.
\]
Letting $r\downarrow0$ gives a contradiction.
\end{proof}

\begin{proof}[Proof of Theorem \ref{thm:alpha-main}]
	The conclusion follows directly from Theorem \ref{thm4.4}.
\end{proof}

\section*{Declarations}
\noindent\textbf{Ethical Approval}

The declaration for ethical approval is not applicable.

\noindent\textbf{Competing interests}

The authors declare no competing interest.

\noindent\textbf{Availability of data and materials}

Data sharing is not applicable to this article as no datasets were generated or analysed during the current study.

\bibliographystyle{scm-num}
\bibliography{ref3}
\end{document}